\documentclass[10pt]{amsart}
\usepackage[leqno]{amsmath}
\usepackage{amssymb,latexsym,soul,cite,amsthm,color,enumitem,graphicx,mathtools,microtype,accents}
\usepackage[colorlinks=true,urlcolor=bronze,citecolor=bronze,linkcolor=bronze,linktocpage,pdfpagelabels,bookmarksnumbered,bookmarksopen]{hyperref}
\definecolor{bronze}{rgb}{0.8, 0.5, 0.2}
\usepackage[english]{babel}
\usepackage[left=2.5cm,right=2.5cm,top=2.5cm,bottom=2.5cm]{geometry}

\numberwithin{equation}{section}

\newtheorem{theorem}{Theorem}[section]
\theoremstyle{plain}
\newtheorem{lemma}[theorem]{Lemma}
\theoremstyle{plain}
\newtheorem{proposition}[theorem]{Proposition}
\theoremstyle{plain}

\theoremstyle{definition}
\newtheorem{remark}[theorem]{Remark}
\newtheorem{example}[theorem]{Example}

\newcommand{\N}{{\mathbb N}}

\newcommand{\R}{{\mathbb R}}
\newcommand{\eps}{\varepsilon}
\newcommand{\beq}{\begin{equation}}
\newcommand{\eeq}{\end{equation}}
\renewcommand{\le}{\leqslant}
\renewcommand{\ge}{\geqslant}

\newcommand{\w}{W^{s,p}_0(\Omega)}
\newcommand{\fpl}{(-\Delta)_p^s\,}
\newcommand{\cs}{C^0_s(\overline\Omega)}
\newcommand{\ds}{{\rm d}_\Omega^s}

\makeatletter
\newcommand{\leqnomode}{\tagsleft@true}
\newcommand{\reqnomode}{\tagsleft@false}
\makeatother

\newenvironment{enumroman}{\begin{enumerate}

}{\end{enumerate}}

\title[Positive solutions for fractional $p$-Laplacian]{Positive solutions for the fractional $p$-Laplacian\\ via mixed topological and variational methods}

\author[A.\ Iannizzotto]{Antonio Iannizzotto}

\address[A.\ Iannizzotto]{Dipartimento di Matematica e Informatica
\newline\indent
Universit\`a di Cagliari
\newline\indent
Via Ospedale 72, 09124 Cagliari, Italy}
\email{antonio.iannizzotto@unica.it}


\begin{document}

\begin{abstract}
We study a nonlinear, nonlocal Dirichlet problem driven by the degenerate fractional $p$-Laplacian via a combination of topological methods (degree theory for operators of monotone type) and variational methods (critical point theory). We assume local conditions ensuring the existence of sub- and supersolutions. So we prove existence of two positive solutions, in both the coercive and noncoercive cases.
\end{abstract}

\maketitle

\begin{center}
Version of \today\
\end{center}

\section{Introduction}\label{sec1}

\noindent
In the study of nonlinear elliptic partial differential equations, two main approaches have been followed in the past decades: in the {\em variational} approach, solutions of the examined problem are seen as critical points of an energy functional and detected via minimization or min-max schemes; in the {\em topological} approach, solutions are seen as zeros of a nonlinear operator and are found via fixed point theorems or degree theory. The topological approach is more general, as it allows to deal with gradient-depending terms, but it usually provides less precise information about the number of solutions. Classical references on the variational and topological methods for nonlinear equations are \cite{S} and \cite{GD}, respectively.\vskip2pt
\noindent
When quasilinear equations are considered (think for instance of the $p$-Laplacian operator), both variational and topological methods are affected by technical difficulties: no Hilbertian structure is available, energy functionals fail to be twice differentiable, so the Leray-Schauder degree for nonlinear operators does not apply. Thus, the $p$-Laplace equation has been the benchmark for developing both a critical point theory for $C^1$-functionals on Banach spaces, and an effective degree theory for operators of monotone type mapping a Banach space into its dual, as well as useful combinations of the two methods, as in the very interesting paper \cite{MMP1}. See \cite{MMP} for a comprehensive account on such techniques.
\vskip2pt
\noindent
In the last decade, beside nonlinearity, {\em nonlocality} has come into play as a major feature, with the increasing interest in fractional order elliptic operators and the inherent new difficulties, which are mainly related to regularity, maximum and comparison principles, and the boundary behavior of solutions. The semilinear case, dealing with the fractional Laplacian or variants of it, is by now well established, see for instance \cite{MRS} for the variational approach. In the quasilinear case, namely for the fractional $p$-Laplacian, a purely variational approach for existence results was proposed in \cite{ILPS}, based on Morse theory and the spectral theory of \cite{LL}. Much has been achieved since then and, despite missing a full organic theory, we have several partial results providing a basic toolbox for the treatment of equations, at least in the degenerate case, i.e., when the summability exponent is $p>2$ (examples are found in \cite{FI} for the variational method and in \cite{FI1} for the topological method).
\vskip2pt
\noindent
In this paper, in some sense a companion work of \cite{FI1}, we aim at combining the variational and topological approaches, relying on recent regularity and comparison results, in order to prove multiplicity of positive solutions for a fractional $p$-Laplace equations under mild assumptions on the reaction. Precisely, we study the following fractional order nonlinear equation with Dirichlet condition:
\beq\label{dir}
\begin{cases}
\fpl u = f(x,u)& \text{in $\Omega$} \\
u=0 & \text{in $\Omega^c$.}
\end{cases}
\eeq
Here $\Omega\subset\R^N$ ($N\geq2$) is a bounded domain with $C^{1,1}$ boundary, $p \ge 2$, $s\in(0,1)$ s.t.\ $N>ps$, the leading operator is the fractional $p$-Laplacian, defined for all $u:\R^N\to\R$ smooth enough and all $x\in\R^N$ by
\[\fpl u(x)=2\lim_{\eps\to 0^+}\int_{B_\eps^c(x)}\frac{|u(x)-u(y)|^{p-2}(u(x)-u(y))}{|x-y|^{N+ps}}\,dy,\]
which for $p=2$ reduces to the fractional Laplacian and for $s\to 1^-$ converges to the classical $p$-Laplacian (up to multiplicative constants). Finally, $f:\Omega\times\R\to\R$ is a Carath\'eodory mapping subject to several growth conditions. We focus on asymptotically $(p-1)$-linear reactions both at infinity and at zero, i.e., we assume that the quotient
\[t\mapsto\frac{f(\cdot,t)}{t^{p-1}}\]
is bounded (uniformly in $\Omega$) as $t\to\infty,\,0$, respectively. The present work differs from \cite{FI1}, mainly because we do not assume a qualitatively different behavior of the quotient above at $\pm\infty$ and at $0$ (jumping reaction). So we distinguish two cases, roughly speaking:
\begin{itemize}[leftmargin=1cm]
\item[$(a)$] in the {\em coercive} case, the limits of the quotient above lie below the principal eigenvalue of the operator $\fpl$ with nonresonance on a positively measured subset of $\Omega$;
\item[$(b)$] in the {\em noncoercive} case, the same limits lie above the principal eigenvalue.
\end{itemize}
Note that in general we allow for the reaction to have exactly the same behavior at infinity and at zero, unlike in many similar works (i.e., we exclude the so-called concave-convex reactions). On the other hand, we shall assume some {\em local} conditions ensuring the existence of sub- and supersolutions, namely, in case $(a)$ we require that the reaction lies above a power $t^{q-1}$ ($q>p$) near the origin, and in case $(b)$ we require that it is nonpositive at some point $b>0$. Under such hypotheses, in both cases we prove that problem \eqref{dir} admits at least two positive solutions, the first being obtained as a local minimizer of the energy functional, and the second being detected through a degree theoretic argument (see Theorems \ref{coe} and \ref{nco} below). We are inspired by similar results from \cite{MMP1} for the local case ($s=1$). We remark that our results are new even in the semilinear framework $p=2$.
\vskip2pt
\noindent
The structure of the paper is the following: in Section \ref{sec2} we recall the degree theory for operators of class $(S)_+$; in Section \ref{sec3} we collect some general results on the fractional $p$-Laplacian and the related Dirichlet problem; in Section \ref{sec4} we prove our multiplicity result for the coercive case $(a)$; and in Section \ref{sec5} we deal with the noncoercive case $(b)$.
\vskip4pt
\noindent
{\bf Notation.} Throughout the paper, for any $A\subset\R^N$ we shall set $A^c=\R^N\setminus A$. For any two measurable functions $u,v:\Omega\to\R$, $u\le v$ in $\Omega$ will mean that $u(x)\le v(x)$ for a.e.\ $x\in\Omega$ (and similar expressions). The positive (resp., negative) part of $u$ is denoted $u^+$ (resp., $u^-$). Every function $u$ defined in $\Omega$ will be identified with its $0$-extension to $\R^N$. If $X$ is an ordered Banach space, then $X_+$ will denote its non-negative order cone. For all $r\in[1,\infty]$, $\|\cdot\|_r$ denotes the standard norm of $L^r(\Omega)$ (or $L^r(\R^N)$, which will be clear from the context). Moreover, $C$ will denote a positive constant (whose value may change case by case).

\section{Degree theory for $(S)_+$-maps}\label{sec2}

\noindent
For the reader's convenience, we recall here some basic notions and properties of Browder's topological degree for $(S)_+$-maps, introduced in \cite{B1} (we follow the general exposition of \cite[Section 4.3]{MMP}).
\vskip2pt
\noindent
Let $X$ be a separable, reflexive Banach space, $X^*$ be its dual space, and $U\subseteq X$. An operator $A:U\to X^*$ is an $(S)_+$-{\em map} if, for any sequence $(u_n)$ in $X$, $u_n\rightharpoonup u$ in $X$ and
\[\limsup_n\,\langle A(u_n),u_n-u\rangle \le 0\]
imply $u_n\to u$ (strongly) in $X$. Also, $A$ is {\em demicontinuous} if it is strong to weak$^*$ continuous. Note that, if $A$ is a demicontinuous $(S)_+$-map and $B:U\to X^*$ is a completely continuous map, then $A+B$ is a demicontinuous $(S)_+$-map as well. In particular, by classical functional analysis (Troyanski's renorming theorem) there exists a $(S)_+$-homeomorphism $\mathcal{F}:X\to X^*$ (duality map) s.t.\ for all $u\in X$
\[\|\mathcal{F}(u)\|^2 = \|u\|^2 = \langle\mathcal{F}(u),u\rangle.\]
Now consider a triple $(A,U,u^*)$ where $U\subset X$ is a bounded open set, $A:\overline{U}\to X^*$ is a demicontinuous $(S)_+$-map, and $u^*\in X^*\setminus A(\partial U)$. By separability, we can perform a Galerkin type approximation of $X$ by means of an increasing sequence $(X_n)$ of finite-dimensional subspaces. For all $n\in\N$ we set $U_n=U\cap X_n$ and define $A_n:\overline{U}_n\to X^*_n$ by setting for all $u\in\overline{U}_n$, $v\in X_n$
\[\langle A_n(u),v\rangle = \langle A(u),v\rangle\]
(for simplicity, we use the same notation for the duality pairings between $X_n^*$ and $X_n$, and between $X^*$ and $X$, respectively). By \cite[Proposition 4.38]{MMP} Brouwer's degree of the triple $(A_n,U_n,u^*)$ eventually stabilizes as $n\to\infty$, so we can define for some $n\in\N$ big enough
\[{\rm deg}_{(S)_+}(A,U,u^*) = {\rm deg}(A_n,U_n,u^*).\]
The integer-valued map $\deg_{(S_+)}$ inherits the main properties of Brouwer's degree:

\begin{proposition}\label{deg}
{\rm \cite[Theorem 4.42]{MMP}} Let $U \subset X$ be a bounded open set, $A: \overline{U} \to X^*$ be a demicontinuous $(S)_+$-map, $u^* \notin A(\partial U)$. Then:
\begin{enumroman}
\item \label{deg1} (normalization) if $u^* \in \mathcal{F}(U)$, then $\deg_{(S)_+} (\mathcal{F}, U, u^*) = 1$;
\item \label{deg2} (domain additivity) if $U= U_1 \cup U_2$, with $U_1, U_2 \subset X$ nonempty open sets s.t.\ $U_1 \cap U_2 = \emptyset$ and $u^* \notin A(\partial U_1 \cup \partial U_2)$, then
\[\deg_{(S)_+} (A, U, u^*)= \deg_{(S)_+} (A, U_1, u^*) + \deg_{(S)_+} (A, U_2, u^*);\]
\item \label{deg3} (excision) if $C \subset \overline{U}$ is closed s.t.\ $u^* \notin A(C)$, then 
\[\deg_{(S)_+} (A, U \setminus C, u^*)= \deg_{(S)_+} (A, U, u^*);\]
\item \label{deg4} (homotopy invariance) if $h: [0,1] \times \overline{U} \to X^*$ is a $(S)_+$-homotopy s.t.\ $u^* \notin h(t,\partial U)$ for all $t \in [0,1]$, then the function
\[t \mapsto \deg_{(S)_+} (h(t,\cdot), U, u^*)\]
is constant in $[0,1]$;
\item \label{deg5} (solution) if $\deg_{(S)_+} (A, U, u^*) \neq 0$, then there exists $u \in U$ s.t.\ $A(u)=u^*$;
\item \label{deg6} (boundary dependence) if $B: \overline{U} \to X^*$ is a demicontinuous $(S)_+$-map s.t.\ $A(u)=B(u)$ for all $u \in \partial U$, then
\[\deg_{(S)_+} (A, U, u^*)= \deg_{(S)_+} (B, U, u^*).\]
\end{enumroman}
\end{proposition}

\noindent
For our purposes, the most important property is \ref{deg4}. We recall that a $(S)_+$-{\em homotopy} is a map $h:[0,1]\times\overline{U}\to X^*$ s.t.\ $t_n \to t$ in $[0,1]$, $u_n \rightharpoonup u$ in $X$, and
\[\limsup_n\,\langle h(t_n,u_n), u_n-u\rangle \le 0\] 
imply $u_n\to u$ in $X$ and $h(t_n,u_n) \rightharpoonup h(t,u)$ in $X^*$. For instance, if $A,B:\overline{U}\to X^*$ are demicontinuous $(S)_+$-maps, then by \cite[Proposition 4.41]{MMP} we can define a $(S)_+$-homotopy as follows:
\[h(t,u) = (1-t)A(u)+tB(u).\]
The bridge between variational and topological methods is represented by the case of a functional $\Phi\in C^1(X)$, whose G\^ateaux derivative $\Phi':X\to X^*$ is a (demi-) continuous $(S)_+$-map ({\em potential operator}). We set
\[K(\Phi) = \big\{u\in X:\,\Phi'(u)=0\big\}.\]
In such cases, the degree of $\Phi'$ in some special sets is related to the local and asymptotic behavior of $\Phi$, respectively, as first proved in \cite{R} (for the Leray-Schauder degree):

\begin{proposition}\label{rab}
{\rm \cite[Corollaries 4.46, 4.49]{MMP}} Let $\Phi\in C^1(X)$ be a functional s.t.\ $\Phi':X\to X^*$ is a continuous $(S)_+$-map:
\begin{enumroman}
\item\label{rab1} if $u_0\in X$ is a local minimizer and an isolated critical point of $\Phi$, then for all $\rho>0$ small enough
\[{\rm deg}_{(S)_+}(\Phi',B_\rho(u_0),0) = 1;\]
\item\label{rab2} if $\Phi$ is coercive and $K(\Phi)$ is bounded, then for all $R>0$ big enough
\[{\rm deg}_{(S)_+}(\Phi',B_R(0),0) = 1.\]
\end{enumroman}
\end{proposition}

\section{The Dirichlet problem for the fractional $p$-Laplacian}\label{sec3}

\noindent
Here we recall some basic features of the existing theory about problem \eqref{dir}. First, for all open $\Omega\subseteq\R^N$ and all measurable $u:\Omega\to\R$ we define the Gagliardo seminorm
\[[u]_{s,p,\Omega} = \Big(\iint_{\Omega\times\Omega}\frac{|u(x)-u(y)|^p}{|x-y|^{N+ps}}\,dx\,dy\Big)^\frac{1}{p}.\]
We introduce the fractional Sobolev spaces (see \cite{DPV} for details)
\[W^{s,p}(\Omega) = \big\{u\in L^p(\Omega):\,[u]_{s,p,\Omega}<\infty\big\},\]
\[\widetilde{W}^{s,p}(\Omega) = \Big\{u\in L^p_{\rm loc}(\R^N):\,u\in W^{s,p}(\Omega') \ \text{for some $\Omega'\Supset\Omega$ and} \ \int_{\R^N}\frac{|u(x)|^{p-1}}{(1+|x|)^{N+ps}}\,dx<\infty\Big\},\]
\[\w = \big\{u\in W^{s,p}(\R^N):\,u=0 \ \text{in $\Omega^c$}\big\}.\]
Clearly $\w\subset\widetilde{W}^{s,p}(\Omega)$, and conversely for all $u\in\widetilde{W}^{s,p}(\Omega)$ s.t.\ $u=0$ in $\Omega^c$ we have $u\in\w$, see \cite[Lemma 2.1]{FI}. If $\Omega\subset\R^N$ is a bounded domain with $C^{1,1}$-smooth boundary $\partial\Omega$, then $\w$ is a uniformly convex, separable Banach space with norm
\[\|u\| = [u]_{s,p,\R^N},\]
whose dual space is denoted $W^{-s,p'}(\Omega)$. Also, the embedding $\w\hookrightarrow L^r(\Omega)$ is continuous for all $r\in[1,p^*_s]$ and compact for all $r\in[1,p^*_s)$, where the fractional critical exponent is defined by
\[p^*_s = \frac{Np}{N-ps}.\]
By \cite[Lemma 2.3]{IMS}, we can extend the definition of the fractional $p$-Laplacian to a wider class of functions. For all $u\in\widetilde{W}^{s,p}(\Omega)$, $\varphi\in\w$ set
\[\langle\fpl u,\varphi\rangle = \iint_{\R^N\times\R^N}\frac{|u(x)-u(y)|^{p-2}(u(x)-u(y))(\varphi(x)-\varphi(y))}{|x-y|^{N+ps}}\,dx\,dy.\]
Then we have $\fpl u\in W^{-s,p'}(\Omega)$, and the definition above agrees with the one given in Section \ref{sec1} for $u$ smooth enough (for instance, $u\in C^\infty_c(\Omega)$). The restricted operator $\fpl:\w\to W^{-s,p'}(\Omega)$ is a continuous $(S)_+$-map (in fact, the duality map of $\w$ is $u\mapsto\fpl u/\|u\|^{p-2}$), as well as the gradient of the $C^1$-functional
\[u\mapsto\frac{\|u\|^p}{p}.\]
We recall a useful formula, holding for any $u\in\w$:
\[\|u^\pm\|^p \le \langle\fpl u,\pm u^\pm\rangle.\]
Also, such operator is strictly $(T)$-monotone:

\begin{proposition}\label{stm}
{\rm \cite[proof of Lemma 9]{LL}} Let $u, v \in \widetilde{W}^{s,p}(\Omega)$ s.t.\ $(u-v)^+ \in \w$ satisfy 
\[\langle \fpl u - \fpl v, (u-v)^+\rangle \le 0,\]
then $u \le v$ in $\Omega$.
\end{proposition}

\noindent
Now let us consider problem \eqref{dir}, under the following basic hypothesis:
\begin{itemize}[leftmargin=1cm]
\item[${\bf H}_0$] $f:\Omega\times\R\to\R$ is a Carath\'{e}odory function and there exist $c_0>0$, $r\in(1,p^*_s)$ s.t.\ for a.e.\ $x\in\Omega$ and all $t\in\R$
\[|f(x,t)| \le c_0 (1+|t|^{r-1}).\]
\end{itemize}
We say that $u\in\widetilde{W}^{s,p}(\Omega)$ is a (weak) supersolution of \eqref{dir} if for all $\varphi\in\w_+$
\[\langle\fpl u,\varphi\rangle \ge \int_\Omega f(x,u)\varphi\,dx,\]
and similarly we define a (weak) subsolution. A (weak) solution of \eqref{dir} is both a super- and a subsolution, i.e., a function $u\in\w$ s.t.\ for all $\varphi\in\w$
\[\langle\fpl u,\varphi\rangle = \int_\Omega f(x,u)\varphi\,dx,\]
The definitions of super-, subsolutions and solutions for other Dirichlet problems will be analogous. We have the following a priori bound:

\begin{proposition}\label{apb}
{\rm\cite[Theorem 3.3]{CMS}} Let ${\bf H}_0$ hold, $u\in\w$ be a solution of \eqref{dir}. Then, $u\in L^{\infty}(\Omega)$ with $\|u\|_{\infty} \le C$, for some $C=C(\|u\|)>0$.
\end{proposition}

\noindent
In fractional regularity theory, the following weighted H\"older spaces play a major role. Set for all $x\in\R^N$
\[{\rm d}_\Omega(x) = {\rm dist}(x,\Omega^c),\]
and for all $\alpha\in(0,1)$
\[C^\alpha_s(\overline\Omega) = \Big\{u \in C^0(\overline{\Omega}): \frac{u}{\ds} \ \text{has a $\alpha$-H\"older continuous extension to} \ \overline{\Omega}\Big\},\]
which is a Banach space endowed with the norm
\[\|u\|_{\alpha,s}= \Big\|\frac{u}{\ds}\Big\|_\infty + \sup_{x \neq y} \frac{|u(x)/\ds(x) - u(y)/\ds(y)|}{|x-y|^{\alpha}}.\]
For $\alpha=0$, the space $\cs$ is defined similarly, with $\alpha$-H\"older continuous replaced by continuous and the corresponding norm
\[\|u\|_{0,s} = \Big\|\frac{u}{\ds}\Big\|_\infty.\]
In addition, the interior of the positive order cone of $\cs$ is
\[{\rm int}(\cs_+) = \Big\{u\in\cs:\,\inf_\Omega\frac{u}{\ds}>0\Big\}.\]
By Proposition \ref{apb} and \cite[Theorem 1.1]{IMS1} we have the following global regularity result:

\begin{proposition}\label{reg}
Let ${\bf H}_0$ hold, $u\in\w$ be a solution of \eqref{dir}. Then, $u\in C^\alpha_s(\overline\Omega)$ for some $\alpha\in(0,s]$ (independent of $u$).
\end{proposition}

\noindent
While Proposition \ref{stm} above can be regarded as a weak comparison principle, strong maximum and comparison principles can be stated as follows:

\begin{proposition}\label{scp}
{\rm \cite[Theorems 2.6, 2.7]{IMP}} Let $g\in C^0(\R)\cap BV_{\rm loc}(\R)$:
\begin{enumroman}
\item\label{scp1} if $u\in\widetilde{W}^{s,p}(\Omega)\cap C^0(\overline\Omega)$ satisfies
\[\begin{cases}
\fpl u+g(u) \ge g(0) & \text{in $\Omega$} \\
u \ge 0 & \text{in $\R^N$}
\end{cases}\]
and $u\neq 0$, then
\[\inf_\Omega\frac{u}{\ds}>0;\]
\item\label{scp2} if $u\in\widetilde{W}^{s,p}(\Omega)\cap C^0(\overline\Omega)$, $v\in\w\cap C^0(\overline\Omega)$, $C>0$ satisfy
\[\begin{cases}
\fpl v+g(v) \le \fpl u+g(u) \le C & \text{in $\Omega$} \\
0 < v \le u & \text{in $\Omega$} \\
0 = v \le u & \text{in $\Omega^c$}
\end{cases}\]
and $u\neq v$, then
\[\inf_\Omega\frac{u-v}{\ds} > 0.\]
\end{enumroman}
\end{proposition}

\noindent
Referring to \cite{I} for details, we consider the following nonlinear weighted eigenvalue problem with weight function $m\in L^\infty(\Omega)$:
\beq\label{evp}
\begin{cases}
\fpl u = \lambda m(x)|u|^{p-2}u & \text{in $\Omega$} \\
u = 0 & \text{in $\Omega^c$.}
\end{cases}
\eeq
The following result summarizes the main properties of the principal eigenvalue of \eqref{evp}:

\begin{proposition}\label{pev}
{\rm\cite[Propositions 3.3, 4.2]{I}} Let $m\in L^\infty(\Omega)$ be s.t.\ $m^+\neq 0$. Then, the smallest eigenvalue of \eqref{evp} is
\[\lambda_1(m) = \inf_{u\in\w\setminus\{0\}}\frac{\|u\|^p}{\int_\Omega m(x)|u|^p\,dx} > 0.\]
In addition:
\begin{enumroman}
\item\label{pev1} $\lambda_1(m)$ is attained at a unique positive, normalized eigenfunction $e_1(m)\in {\rm int}(\cs_+)$, while any nonprincipal eigenfunction changes sign in $\Omega$;
\item\label{pev2} for all $\tilde m\in L^\infty(\Omega)$ s.t.\ $m\le\tilde m$ in $\Omega$ and $m\neq\tilde m$, we have
\[\lambda_1(m) > \lambda_1(\tilde m).\]
\end{enumroman}
\end{proposition}

\noindent
In particular, we will write $\lambda_1=\lambda_1(1)$, $e_1=e_1(1)$. We recall two technical results, related to the eigenvalue $\lambda_1$. The first is of variational nature:

\begin{proposition}\label{var}
{\rm \cite[Lemma 2.7]{IL}} Let $\xi\in L^\infty(\Omega)$ be s.t.\ $\xi\le\lambda_1$ in $\Omega$, and $\xi\neq\lambda_1$. Then, there exists $\sigma>0$ s.t.\ for all $u\in\w$
\[\|u\|^p-\int_\Omega \xi(x)|u|^p\,dx \ge \sigma\|u\|^p.\]
\end{proposition}

\noindent
The second is a nonlocal version of the antimaximum principle of \cite{GGP}:

\begin{proposition}\label{amp}
{\rm \cite[Lemma 3.9]{FI1}} Let $m,\beta\in L^\infty(\Omega)_+\setminus\{0\}$, $\lambda\ge\lambda_1(m)$, $u\in\w$ solve
\[\begin{cases}
\fpl u = \lambda m(x)|u|^{p-2}u+\beta(x) & \text{in $\Omega$} \\
u = 0 & \text{in $\Omega^c$.}
\end{cases}\]
Then, $u^-\neq 0$.
\end{proposition}

\noindent
Now we focus on the variational formulation of problem \eqref{dir}. First, set for all $(x,t)\in\Omega\times\R$
\[F(x,t) = \int_0^t f(x,\tau)\,d\tau.\]
Then, set for all $u\in\w$
\[\Phi(u) = \frac{\|u\|^p}{p}-\int_\Omega F(x,u)\,dx.\]
By virtue of ${\bf H}_0$, it is easily seen that $\Phi\in C^1(\w)$. For all $u\in\w$ we have
\[\Phi'(u) = \fpl u-N_f(u),\]
where $N_f:\w\to W^{-s,p'}(\Omega)$ is the completely continuous Nemitskii operator defined for all $u,\varphi\in\w$ by
\[\langle N_f(u),\varphi\rangle = \int_\Omega f(x,u)\varphi\,dx.\]
So, as seen in Section \ref{sec2}, $\Phi':\w\to W^{-s,p'}(\Omega)$ is a (demi)-continuous $(S)_+$-map. In addition, $\Phi$ is sequentially weakly l.s.c.\ and satisfies a bounded $(PS)$-condition, i.e., whenever $(u_n)$ is a bounded sequence in $\w$ s.t.\ $|\Phi(u_n)|\le C$ for all $n\in\N$ and $\Phi'(u_n)\to 0$ in $W^{-s,p'}(\Omega)$, then up to a subsequence $u_n\to u$ in $\w$.
\vskip2pt
\noindent
Finally, we recall the equivalence between Sobolev and H\"older local minimizers of $\Phi$:

\begin{proposition}\label{svh}
{\rm \cite[Theorem 1.1]{IMS2}} Let ${\bf H}_0$ hold, $u\in\w$. Then, the following are equivalent:
\begin{enumroman}
\item\label{svh1} there exists $\sigma>0$ s.t.\ $\Phi(u+v)\ge\Phi(u)$ for all $v\in\w\cap C_s^0(\overline\Omega)$, $\|v\|_{0,s}\le\sigma$;
\item\label{svh2} there exists $\rho>0$ s.t.\ $\Phi(u+v)\ge\Phi(u)$ for all $v\in\w$, $\|v\| \le\rho$.
\end{enumroman}
\end{proposition}

\begin{remark}\label{sin}
All results of this section also hold in the singular case $p\in (1,2)$, but the regularity result Proposition \ref{reg}, which has only been proved for the degenerate case $p\ge 2$ so far, and consequently Proposition \ref{svh}.
\end{remark}

\section{Coercive case}\label{sec4}

\noindent
In this section we deal with the case when the limits of the quotient
\[t\mapsto\frac{f(x,t)}{t^{p-1}}\]
both at zero and infinity lie below the principal eigenvalue $\lambda_1$ without resonance (in fact we will assume a slightly more general condition). This case is called {\em coercive}, since the energy functional corresponding to problem \eqref{dir} tends to infinity as $\|u\|\to\infty$. In order to detect a positive subsolution, we will make use of an auxiliary Dirichlet problem for the {\em $p$-fractional Lane-Emden equation}:
\beq\label{lee}
\begin{cases}
\fpl v = |v|^{q-2}v & \text{in $\Omega$} \\
v = 0 & \text{in $\Omega^c$.}
\end{cases}
\eeq
Here $q\in(p,p^*_s)$, and recalling that $\w\hookrightarrow L^q(\Omega)$ we set
\[c_q = \inf_{u\in\w\setminus\{0\}}\frac{\|u\|}{\|u\|_q} > 0.\]
The following technical result is a nonlocal version of \cite[Theorem 1]{O}:

\begin{lemma}\label{ota}
Let $q\in(p,p^*_s)$. Then problem \eqref{lee} has at least one solution $v_q\in{\rm int}(\cs_+)$ s.t.\
\[\|v_q\|^p = \|v_q\|_q^q = c_q^\frac{pq}{q-p}.\]
\end{lemma}
\begin{proof}
Set
\[\mathcal{S}_q = \big\{v\in\w:\,\|v\|_q^q=c_q^\frac{pq}{q-p}\big\}.\]
By the compact embedding $\w\hookrightarrow L^q(\Omega)$, it is easily seen that $\mathcal{S}_q$ is sequentially weakly closed as a subset of $\w$. Also for all $v\in\mathcal{S}_q$ we have
\beq\label{ota1}
\|v\|^p \ge c_q^p\|v\|_q^p = c_q^\frac{pq}{q-p}.
\eeq
By definition of $c_q$, there exists a sequence $(u_n)$ in $\w\setminus\{0\}$ s.t.\ 
\[\frac{\|u_n\|}{\|u_n\|_q} \to c_q.\]
By replacing if necessary $u_n$ with $|u_n|$, we may assume $u_n\ge 0$ in $\Omega$, for all $n\in\N$. Set
\[v_n = c_q^\frac{p}{q-p}\frac{u_n}{\|u_n\|_q} \in \mathcal{S}_q,\]
then $\|v_n\|^p\to c_q^\frac{pq}{q-p}$. In particular, $(v_n)$ is bounded in $\w$. Passing if necessary to a subsequence, we have $v_n\rightharpoonup v_q$ in $\w$, $v_n\to v_q$ in $L^q(\Omega)$, and $v_n(x)\to v_q(x)$ for a.e.\ $x\in\Omega$. So $v_q\in\mathcal{S}_q$ and $v_q\ge 0$ in $\Omega$. Also,
\[\|v_q\|^p \le \liminf_n\|v_n\|^p = c_q^\frac{pq}{q-p},\]
which along with \eqref{ota1} gives
\[\|v_q\|^p = \|v_q\|_q^q = c_q^\frac{pq}{q-p}.\]
Thus, $v_q$ is a minimizer of the functional $v\mapsto\|v\|^p$ restricted the $C^1$-manifold $\mathcal{S}_q$. By Lagrange's multipliers rule, there exists $\mu\in\R$ s.t.\ for all $\varphi\in\w$
\[\langle\fpl v_q,\varphi\rangle = \mu \int_\Omega v_q^{q-1}\varphi\,dx.\]
Testing the relation above with $v_q\in\w$ we get
\[\|v_q\|^p = \mu\|v_q\|_q^q,\]
which implies $\mu=1$. So $v_q$ is a non-negative weak solution of \eqref{lee}, in the sense of Section \ref{sec3}. Clearly the reaction
\[(x,t) \mapsto |t|^{q-2}t\]
satisfies ${\bf H}_0$, so by Proposition \ref{reg} we have $v_q\in C^\alpha_s(\overline\Omega)$. Now apply Proposition \ref{scp} \ref{scp1} (with $g(t)=-|t|^{q-2}t$) to find $v_q\in{\rm int}(\cs_+)$.
\end{proof}

\noindent
Our hypotheses on the reaction $f$, in the present case, are the following:
\begin{itemize}[leftmargin=1cm]
\item[${\bf H}_1$] $f:\Omega\times\R\to\R$ is a Carath\'eodory function and
\begin{enumroman}
\item\label{h11} for all $M>0$ there exists $a_M\in L^\infty(\Omega)_+$ s.t.\ for a.e.\ $x\in\Omega$ and all $t\in[0,M]$
\[0 \le f(x,t) \le a_M(x);\]
\item\label{h12} there exists $\theta\in L^\infty(\Omega)_+$ s.t.\ $\theta\le\lambda_1$ in $\Omega$, $\theta\neq\lambda_1$, and uniformly for a.e.\ $x\in\Omega$
\[\limsup_{t\to\infty}\frac{f(x,t)}{t^{p-1}} \le \theta(x);\]
\item\label{h13} there exists $\eta\in L^\infty(\Omega)_+$ s.t.\ $\eta\le\lambda_1$ in $\Omega$, $\eta\neq\lambda_1$, and uniformly for a.e.\ $x\in\Omega$
\[\limsup_{t\to 0^+}\frac{f(x,t)}{t^{p-1}} \le \eta(x);\]
\item\label{h14} there exists $q\in(p,p^*_s)$ s.t.\ for a.e.\ $x\in\Omega$ and all $t>0$
\[f(x,t) > \min\big\{t,\,\|v_q\|_\infty\big\}^{q-1}.\]
\end{enumroman}
\end{itemize}
Hypotheses ${\bf H}_1$ above only concern the behavior of $f(x,\cdot)$ in the positive semiaxis, but since we are seeking positive solutions, we can use ${\bf H}_1$ \ref{h13} and set for all $x\in\Omega$, $t\in\R^-$
\[f(x,t) = 0.\]
Clearly ${\bf H}_1$ imply ${\bf H}_0$, hence all results of Section \ref{sec3} apply. Hypotheses ${\bf H}_1$ \ref{h12} \ref{h13} conjure a $(p-1)$-sublinear growth at infinity and a $(p-1)$-superlinear growth at the origin, with the principal eigenvalue $\lambda_1>0$ (defined in Proposition \ref{pev}) as a threshold slope. Also, $v_q\in{\rm int}(\cs_+)$ in ${\bf H}_1$ \ref{h14} is defined by Lemma \ref{ota}.

\begin{example}
Assume that, for convenient $\Omega$, $p$, $q$ we have $\lambda_1>\|v_q\|_\infty^{q-p}$. Then we can find $\lambda\in(0,\lambda_1)$, $\mu>1$ s.t.\ $\lambda>\mu\|v_q\|_\infty^{q-p}$. So set for all $t\in\R$
\[f(t) = \begin{cases}
0 & \text{if $t\le 0$} \\
\mu t^{q-1} & \text{if $0<t\le\|v_q\|_\infty$} \\
\lambda t^{p-1}+\mu\|v_q\|_\infty^{q-1}-\lambda\|v_q\|_\infty^{p-1} & \text{if $t>\|v_q\|_\infty$.}
\end{cases}\]
Elementary calculations show that $f\in C^0(\R)$ satisfies ${\bf H}_1$.
\end{example}

\noindent
Our multiplicity result for the coercive case is the following, extending \cite[Theorem 11.13]{MMP} to the nonlocal framework (see also the first part of \cite[Theorem 1]{MMP1}):

\begin{theorem}\label{coe}
Let ${\bf H}_1$ hold. Then, problem \eqref{dir} has at least two solutions $u_1,u_2\in{\rm int}(\cs_+)$.
\end{theorem}
\begin{proof}
By ${\bf H}_1$ (and the extension of $f$ to the negative semiaxis) we can define the primitive $F:\Omega\times\R\to\R$ and the functional $\Phi\in C^1(\w)$ as in Section \ref{sec3}.
\vskip2pt
\noindent
Our first claim is that $0$ is a local minimizer of $\Phi$. Let $\sigma>0$ be as in Proposition \ref{var} (with $\xi=\eta$ from ${\bf H}_1$ \ref{h13}). Fix $\eps\in(0,\sigma\lambda_1)$, then by ${\bf H}_1$ \ref{h13} and de l'H\^opital's rule there exists $\delta>0$ s.t.\ for a.e.\ $x\in\Omega$ and all $t\in[0,\delta]$
\[F(x,t) \le \frac{\eta(x)+\eps}{p}t^p.\]
By Proposition \ref{var} and the variational characterization of $\lambda_1$, for all $u\in\w\cap\cs$ with $\|u\|_\infty\le\delta$ we have
\begin{align*}
\Phi(u) &\ge \frac{\|u\|^p}{p}-\int_\Omega\frac{\eta(x)+\eps}{p}(u^+)^p\,dx \\
&\ge \frac{1}{p}\Big(\|u\|^p-\int_\Omega\eta(x)|u|^p\,dx\Big)-\frac{\eps}{p}\|u\|_p^p \\
&\ge \Big(\sigma-\frac{\eps}{\lambda_1}\Big)\frac{\|u\|^p}{p} \ge 0.
\end{align*}
Since $\Phi(0)=0$ and $\cs\hookrightarrow L^\infty(\Omega)$, we see that $0$ is a local minimizer in $\cs$ for $\Phi$. By Proposition \ref{svh}, it is such as well in $\w$, as claimed.
\vskip2pt
\noindent
Now let $v_q\in{\rm int}(\cs_+)$ be as in Lemma \ref{ota}. By ${\bf H}_1$ \ref{h14} we have for all $\varphi\in\w_+\setminus\{0\}$
\[\langle\fpl v_q,\varphi\rangle = \int_\Omega v_q^{q-1}\varphi\,dx < \int_\Omega f(x,v_q)\varphi\,dx,\]
so $v_q$ is a (strict) subsolution of \eqref{dir}. Set for all $(x,t)\in\Omega\times\R$
\[\hat f(x,t) = f\big(x,\max\{t,v_q(x)\}\big), \ \hat F(x,t) = \int_0^t \hat f(x,\tau)\,d\tau,\]
so $\hat f:\Omega\times\R\to\R$ satisfies ${\bf H}_0$. Then set for all $u\in\w$
\[\hat\Phi(u) = \frac{\|u\|^p}{p}-\int_\Omega\hat F(x,u)\,dx.\]
As in Section \ref{sec3} we see that $\hat\Phi\in C^1(\w)$ and is sequentially weakly l.s.c. Moreover, $\hat\Phi$ is coercive. To see this, let $\sigma>0$ be as in Proposition \ref{var} (this time with $\xi=\theta$ from ${\bf H}_1$ \ref{h12}). Fix $\eps\in(0,\sigma\lambda_1)$, then by ${\bf H}_1$ \ref{h12} there exists $M>\|v_q\|_\infty$ s.t.\ for a.e.\ $x\in\Omega$ and all $t\ge M$
\[f(x,t) \le (\theta(x)+\eps)t^{p-1}.\]
Further, by ${\bf H}_1$ \ref{h11} we can find $a_M\in L^\infty(\Omega)_+$ s.t.\ for a.e.\ $x\in\Omega$ and all $t\ge M$
\begin{align*}
\hat F(x,t) &= \int_0^{v_q}f(x,v_q)\,d\tau+\int_{v_q}^M f(x,\tau)\,d\tau+\int_M^t f(x,\tau)\,d\tau \\
&\le a_M(x)M+\frac{\theta(x)+\eps}{p}(t^p-M^p).
\end{align*}
Using again ${\bf H}_1$ \ref{h11} we find $C>0$ s.t.\ for a.e.\ $x\in\Omega$ and all $t\in\R$
\[\hat F(x,t) \le \frac{\theta(x)+\eps}{p}(t^+)^p+C.\]
So, by Proposition \ref{var} and the variational characterization of $\lambda_1$, we have for all $u\in\w$
\begin{align*}
\hat\Phi(u) &\ge \frac{\|u\|^p}{p}-\int_\Omega\Big(\frac{\theta(x)+\eps}{p}(u^+)^p+C\Big)\,dx \\
&\ge \frac{1}{p}\Big(\|u\|^p-\int_\Omega\theta(x)|u|^p\,dx\Big)-\frac{\eps}{p}\|u\|_p^p-C \\
&\ge \Big(\sigma-\frac{\eps}{\lambda_1}\Big)\frac{\|u\|^p}{p}-C,
\end{align*}
and the latter tends to $\infty$ as $\|u\|\to\infty$. Thus, there exists $u_1\in\w$ s.t.\
\beq\label{coe1}
\hat\Phi(u_1) = \inf_{u\in\w}\hat\Phi(u).
\eeq
In particular, we have weakly in $\Omega$
\beq\label{coe2}
\fpl u_1 = \hat f(x,u_1).
\eeq
Testing \eqref{lee} and \eqref{coe2} with $(v_q-u_1)^+\in\w$, and using ${\bf H}_1$ \ref{h14}, we have
\[\langle\fpl v_q-\fpl u_1,(v_q-u_1)^+\rangle = \int_{\{u_1<v_q\}}\big(v_q^{q-1}-f(x,v_q)\big)(v_q-u_1)\,dx \le 0.\]
By Proposition \ref{stm} we have $u_1\ge v_q$ in $\Omega$. By construction, we may replace $\hat f$ with $f$ in \eqref{coe2} and see that $u_1$ solves \eqref{dir}. By Proposition \ref{reg}, then, we have $u_1\in C^\alpha_s(\overline\Omega)$. By ${\bf H}_1$ \ref{h14} we have weakly in $\Omega$
\[\fpl v_q = v_q^{q-1} \le f(x,u_1) = \fpl u_1.\]
Besides, since $v_q$ is a strict subsolution of \eqref{dir}, we have $u_1\neq v_q$. So Proposition \ref{scp} \ref{scp2} (with $g(t)=0$) implies
\[u_1-v_q \in {\rm int}(\cs_+),\]
in particular $u_1\in{\rm int}(\cs_+)$. Now set
\[V = \big\{v\in\w\cap\cs:\,v-v_q\in{\rm int}(\cs_+)\big\}.\]
Since $u_1-v_q\in{\rm int}(\cs_+)$, there exists $\sigma>0$ s.t.\ for all $v\in\w\cap\cs$ with $\|v-u_1\|_{0,s}\le\sigma$ we have $v\in V$. By \eqref{coe1} we have for all $v\in V$
\[\Phi(v) = \hat\Phi(v) \ge \hat\Phi(u_1) = \Phi(u_1),\]
hence $u_1$ is a local minimizer of $\Phi$ in $\cs$. By Proposition \ref{svh}, it is as well a local minimizer of $\Phi$ in $\w$.
\vskip2pt
\noindent
There remains to prove that $\Phi$ has a further critical point, beside $0$ and $u_1$. With this aim in mind, we distinguish two cases:
\begin{itemize}[leftmargin=1cm]
\item[$(a)$] If either $0$ or $u_1$ is not an isolated critical point of $\Phi$, then clearly $\Phi$ has infinitely many critical points.
\item[$(b)$] If both $0$ and $u_1$ are isolated critical points of $\Phi$, then in particular they are {\em strict} local minimizers. We then apply Proposition \ref{rab} \ref{rab1} and find for all $\rho>0$ small enough
\beq\label{coe3}
{\rm deg}_{(S)_+}(\Phi',B_\rho(0),0) = {\rm deg}_{(S)_+}(\Phi',B_\rho(u_1),0) = 1.
\eeq
Also, arguing as we did above with $\hat\Phi$, we see that $\Phi$ is coercive in $\w$. So, by Proposition \ref{rab} \ref{rab2} we have for all $R>0$ big enough
\beq\label{coe4}
{\rm deg}_{(S)_+}(\Phi',B_R(0),0) = 1.
\eeq
We choose $0<\rho<R$ in the relations above so that
\[\overline{B}_\rho(0)\cap\overline{B}_\rho(u_1) = \emptyset, \ \overline{B}_\rho(0)\cup\overline{B}_\rho(u_1) \subset B_R(0).\]
Using \eqref{coe3} \eqref{coe4} and Proposition \ref{deg} \ref{deg2} (domain additivity) twice, we get
\begin{align*}
1 &= {\rm deg}_{(S)_+}\big(\Phi',B_R(0),0) \\
&= {\rm deg}_{(S)_+}\big(\Phi',B_\rho(0),0)+{\rm deg}_{(S)_+}\big(\Phi',B_\rho(u_1),0)+{\rm deg}_{(S)_+}\big(\Phi',B_R(0)\setminus(\overline{B}_\rho(0)\cup\overline{B}_\rho(u_1)),0\big) \\
&= 2+{\rm deg}_{(S)_+}\big(\Phi',B_R(0)\setminus(\overline{B}_\rho(0)\cup\overline{B}_\rho(u_1)),0\big),
\end{align*}
which rephrases as
\[{\rm deg}_{(S)_+}\big(\Phi',B_R(0)\setminus(\overline{B}_\rho(0)\cup\overline{B}_\rho(u_1)),0\big) = -1.\]
Then, by Proposition \ref{deg} \ref{deg5} (solution), there exists $u_2\in B_R(0)\setminus(\overline{B}_\rho(0)\cup\overline{B}_\rho(u_1))$ s.t.\
\[\Phi'(u_2) = 0.\]
\end{itemize}
In either case, we end up with two nontrivial critical points $u_1,u_2\in K(\Phi)\setminus\{0\}$. We already know that $u_1\in{\rm int}(\cs_+)$ solves \eqref{dir}.
\vskip2pt
\noindent
So, let us consider $u_2$. By Proposition \ref{reg} we see that $u_2\in C^\alpha_s(\overline\Omega)$ solves \eqref{dir}. Testing with $-u_2^-\in\w$ and using ${\bf H}_1$ \ref{h11} we have
\begin{align*}
\|u_2^-\|^p &\le \langle\fpl u_2,-u_2^-\rangle \\
&= \int_{\{u_2<0\}} f(x,u_2)u_2\,dx \le 0,
\end{align*}
hence $u_2\ge 0$ in $\Omega$. Now apply Proposition \ref{scp} \ref{scp1} (with $g(t)=0$) and recall that $u_2\neq 0$, to conclude that $u_2\in{\rm int}(\cs_+)$, which ends the proof.
\end{proof}

\begin{remark}\label{mpt}
In the proof of Theorem \ref{coe}, the degree-theoretical argument used to detect a second positive solution can be replaced by an equivalent variational one, based on the mountain pass theorem (see \cite[Theorem 5.40]{MMP}). Also, we note that under hypotheses symmetric to ${\bf H}_1$ on the negative semiaxis, we can prove existence of two negative solutions.
\end{remark}

\section{Noncoercive case}\label{sec5}

\noindent
In this section we consider the more delicate case in which the limits of
\[t \mapsto \frac{f(x,t)}{t^{p-1}},\]
both for $t\to\infty,0^+$ lie above the principal eigenvalue $\lambda_1$. Such asymptotic behavior prevents both coercivity of the energy functional and the existence of a local minimum at $0$. So, in the present case the use of degree theory is more meaningful.
\vskip2pt
\noindent
Our hypotheses on the reaction $f$ are the following:
\begin{itemize}[leftmargin=1cm]
\item[${\bf H}_2$] $f:\Omega\times\R\to\R$ is a Carath\'eodory function and
\begin{enumroman}
\item\label{h21} for all $M>0$ there exists $a_M\in L^\infty(\Omega)_+$ s.t.\ for a.e.\ $x\in\Omega$ and all $t\in[0,M]$
\[|f(x,t)| \le a_M(x);\]
\item\label{h22} there exist $\theta_1,\theta_2\in L^\infty(\Omega)_+$ s.t.\ $\theta_1\ge\lambda_1$ in $\Omega$, $\theta_1\neq\lambda_1$, and uniformly for a.e.\ $x\in\Omega$
\[\theta_1(x) \le \liminf_{t\to\infty}\frac{f(x,t)}{t^{p-1}} \le \limsup_{t\to\infty}\frac{f(x,t)}{t^{p-1}} \le \theta_2(x);\]
\item\label{h23} there exist $\eta_1,\eta_2\in L^\infty(\Omega)_+$ s.t.\ $\eta_1\ge\lambda_1$ in $\Omega$, $\eta_1\neq\lambda_1$, and uniformly for a.e.\ $x\in\Omega$
\[\eta_1(x) \le \liminf_{t\to 0^+}\frac{f(x,t)}{t^{p-1}} \le \limsup_{t\to 0^+}\frac{f(x,t)}{t^{p-1}} \le \eta_2(x);\]
\item\label{h24} there exist $b,K>0$ s.t.\ uniformly for a.e.\ $x\in\Omega$
\[\limsup_{t\to b^-}\frac{f(x,t)}{(b-t)^{p-1}} \le K.\]
\end{enumroman}
\end{itemize}
As in Section \ref{sec4}, by ${\bf H}_2$ \ref{h23} we may set for all $x\in\Omega$, $t\le 0$
\[f(x,t) = 0.\]
So ${\bf H}_2$ imply ${\bf H}_0$. By ${\bf H}_2$ \ref{h22} \ref{h23}, $f(x,\cdot)$ has precisely a $(p-1)$-linear growth both at $0$ and at $\infty$, with limits slopes above the principal eigenvalue $\lambda_1$ (Proposition \ref{pev}) without resonance. Also, hypothesis ${\bf H}_2$ \ref{h24} implies that the constant $b$ is a supersolution of \eqref{dir}.

\begin{example}
Define an autonomous reaction $f\in C^0(\R)$ by setting for all $t\ge 0$
\[f(t) = \mu t^{p-1}-\gamma\arctan(t^{q-1}),\]
with $q>p$, $\mu>\lambda_1$, and $\gamma>0$. Elementary calculus shows that $f$  satisfies ${\bf H}_2$ \ref{h21} -- \ref{h23}. Moreover, if $\gamma$ is big enough (depending on $\mu$, $p$, and $q$), then $f$ becomes negative at some $b>0$, hence ${\bf H}_2$ \ref{h24} is satisfied as well.
\end{example}

\noindent
Our multiplicity theorem for the noncoercive case, extending \cite[Theorem 11.15]{MMP} to the nonlocal framework (see also the second part of \cite[Theorem 1]{MMP1}), is the following:

\begin{theorem}\label{nco}
Let ${\bf H}_2$ hold. Then, problem \eqref{dir} has at least two solutions $u_1,u_2\in{\rm int}(\cs_+)$.
\end{theorem}
\begin{proof}
Since ${\bf H}_0$ holds, we can define $\Phi\in C^1(\w)$ as in Section \ref{sec3}. First, we note that by ${\bf H}_2$ \ref{h22} we have $f(x,0)=0$ for a.e.\ $x\in\Omega$, so $0$ is a critical point of $\Phi$.
\vskip2pt
\noindent
As anticipated above, the constant function $b\in\widetilde{W}^{s,p}(\Omega)$ is a (strict) supersolution of \eqref{dir}. Indeed, by ${\bf H}_2$ \ref{h24} we have $f(x,b)\le 0$ for a.e.\ $x\in\Omega$, hence for all $\varphi\in\w_+$
\[\langle\fpl b,\varphi\rangle = 0 \ge \int_\Omega f(x,b)\varphi\,dx.\]
Set for all $(x,t)\in\Omega\times\R$
\[\check f(x,t) = f\big(x,\min\{t,b\}\big), \ \check F(x,t) = \int_0^t\check f(x,\tau)\,d\tau.\]
Then, $\check f:\Omega\times\R\to\R$ satisfies ${\bf H}_0$. So set for all $u\in\w$
\[\check\Phi(u) = \frac{\|u\|^p}{p}-\int_\Omega\check F(x,u)\,dx.\]
As seen in Section \ref{sec3}, $\check\Phi\in C^1(\w)$ is sequentially weakly l.s.c.\ In addition, such functional is coercive. Indeed, by ${\bf H}_2$ \ref{h21} and the definition of $\check f$ we have for a.e.\ $x\in\Omega$ and all $t\ge b$
\[\check F(x,t) = \int_0^b f(x,\tau)\,d\tau+\int_b^t f(x,b)\,d\tau \le a_b(x)b\]
(recall that $f(\cdot,b)\le 0$ in $\Omega$), which implies for a.e.\ $x\in\Omega$ and all $t\in\R$
\[\check F(x,t) \le C.\]
So, by the continuous embedding $\w\hookrightarrow L^1(\Omega)$ we have for all $u\in\w$
\[\check\Phi(u) \ge \frac{\|u\|^p}{p}-C,\]
and the latter tends to $\infty$ as $\|u\|\to\infty$. Thus, there exists $u_1\in\w$ s.t.\
\beq\label{nco1}
\check\Phi(u_1) = \inf_{u\in\w}\check\Phi(u).
\eeq
By \eqref{nco1} we have weakly in $\Omega$
\beq\label{nco2}
\fpl u_1 = \check f(x,u_1).
\eeq
Since $\check f$ satisfies ${\bf H}_0$, by Proposition \ref{reg} we have $u_1\in C^\alpha_s(\overline\Omega)$. Recalling that $f(\cdot,t)=0$ in $\Omega$ for all $t\le 0$, and testing \eqref{nco2} with $-u_1^-\in\w$ we have
\begin{align*}
\|u_1^-\|^p &\le \langle\fpl u_1,-u_1^-\rangle \\
&= \int_{\{u_1<0\}}f(x,u_1)u_1\,dx = 0,
\end{align*}
hence $u_1\ge 0$ in $\Omega$. On the other hand, by ${\bf H}_2$ \ref{h24} we have $f(\cdot,b)\le 0$ in $\Omega$. So, testing \eqref{nco2} with $(u_1-b)^+\in\w$ we have
\begin{align*}
\langle\fpl u_1-\fpl b,(u_1-b)^+\rangle &= \int_\Omega \check f(x,u_1)(u_1-b)^+\,dx \\
&= \int_{\{u_1>b\}} f(x,b)(u_1-b)\,dx \le 0,
\end{align*}
hence $u_1\le b$ in $\Omega$. This already implies that we can replace $\check f$ with $f$ in \eqref{nco2} and see that $u_1\in C^\alpha_s(\overline\Omega)$ solves \eqref{dir} (see Proposition \ref{reg}). Note that, {\em so far}, $u_1=0$ may occur. Set
\[V = \big\{v\in\w\cap\cs:\,v<b \ \text{in $\Omega$}\big\}.\]
We aim at proving that $V$ is a neighborhood of $u_1$ in the $\cs$-topology, distinguishing two cases:
\begin{itemize}[leftmargin=1cm]
\item[$(a)$] If $u_1=0$, then clearly there exists $\sigma>0$ s.t.\ for all $v\in\w\cap\cs$ with $\|v\|_{0,s}\le\sigma$ we have $v<b$ in $\Omega$, hence $v\in V$.
\item[$(b)$] If $u_1\neq 0$, then we have $u_1>0$ in $\Omega$. To see this, fix $\eps\in(0,\lambda_1)$. By ${\bf H}_2$ \ref{h23} there exists $\delta\in(0,b)$ s.t.\ for a.e.\ $x\in\Omega$ and all $t\in[0,\delta]$
\[f(x,t) > \eps t^{p-1}.\]
Besides, by ${\bf H}_2$ \ref{h21} we have for a.e.\ $x\in\Omega$ and all $t\in[\delta,b]$
\[f(x,t) \ge -a_b(x) \ge -\frac{\|a_b\|_\infty}{\delta^{p-1}}t^{p-1}.\]
Recalling the construction of $\check f$, we find $C>0$ s.t.\ for a.e.\ $x\in\Omega$ and all $t\ge 0$
\[\check f(x,t) \ge -Ct^{p-1}.\]
By \eqref{nco2} we have weakly in $\Omega$
\[\fpl u_1+Cu_1^{p-1} \ge 0,\]
while $u_1\ge 0$ in $\Omega$ and $u_1\neq 0$. By Proposition \ref{scp} \ref{scp1} (with $g(t)=C(t^+)^{p-1}$) we have
\[\inf_\Omega\frac{u_1}{\ds} > 0.\]
In particular, we have $u_1\in{\rm int}(\cs_+)$. Now, let us compare $u_1$ to $b$. By ${\bf H}_2$ \ref{h24}, for all $\eps>0$ there exists $\delta\in(0,b)$ s.t.\ for a.e.\ $x\in\Omega$ and all $t\in[\delta,b]$
\[f(x,t) \le (K+\eps)(b-t)^{p-1}.\]
Besides, by ${\bf H}_2$ \ref{h21} we have for a.e.\ $x\in\Omega$ and all $t\in[0,\delta]$
\[f(x,t) \le a_\delta(x) \le \frac{\|a_\delta\|_\infty}{(b-\delta)^{p-1}}(b-t)^{p-1}.\]
So we can find $C>0$ s.t.\ for a.e.\ $x\in\Omega$ and all $t\in[0,b]$
\[\check f(x,t) \le C(b-t)^{p-1}.\]
Now recall that $0<u_1\le b$ in $\Omega$ and \eqref{nco2} holds, so we have weakly in $\Omega$
\[\fpl u_1-C(b-u_1)^{p-1} \le 0 = \fpl b-C(b-b)^{p-1}.\]
By Proposition \ref{scp} \ref{scp2} (with $g(t)=-C((b-t)^+)^{p-1}$) we have
\[\inf_\Omega\frac{b-u_1}{\ds} > 0.\]
So, we can find $\sigma>0$ s.t.\ for all $v\in\w\cap\cs$ with $\|v-u_1\|_{0,s}\le\sigma$ we have $v\in V$.
\end{itemize}
By \eqref{nco1} and the estimates above, for all $v\in V$ we have
\[\Phi(v) = \check\Phi(v) \ge \check\Phi(u_1) = \Phi(u_1).\]
So, $u_1$ is a local minimizer of $\Phi$ in $\cs$. By Proposition \ref{svh}, it is as well a local minimizer of $\Phi$ in $\w$.
\vskip2pt
\noindent
We seek now another critical point of $\Phi$, using degree theory. First, note that by ${\bf H}_2$ \ref{h22} we can find $M,C>0$ s.t.\ for a.e.\ $x\in\Omega$ and all $t\ge M$
\[|f(x,t)| \le Ct^{p-1}.\]
Also, by ${\bf H}_2$ \ref{h21} there exists $a_M\in L^\infty(\Omega)_+$ s.t.\ for a.e.\ $x\in\Omega$ and all $t\in[0,M]$
\[|f(x,t)| \le a_M(x).\]
Recalling that $f(\cdot,t)=0$ in $\Omega$ for all $t\le 0$, we have for a.e.\ $x\in\Omega$ and all $t\in\R$
\beq\label{nco3}
|f(x,t)| \le C(1+|t|^{p-1}),
\eeq
i.e., $f$ satisfies ${\bf H}_0$ (with $r=p$). Define the completely continuous map $N_f:\w\to W^{-s,p'}(\Omega)$ as in Section \ref{sec3}. So, $\Phi':\w\to W^{-s,p'}(\Omega)$ is a continuous $(S)_+$-map. Our first claim is that, for all $R>0$ big enough,
\beq\label{nco4}
{\rm deg}_{(S)_+}(\Phi',B_R(0),0) = 0.
\eeq
To see this, first let $\theta_1,\theta_2\in L^\infty(\Omega)_+$ be as in ${\bf H}_2$ \ref{h22}, then set for all $u,\varphi\in\w$
\[\langle K_\infty(u),\varphi\rangle = \int_\Omega\frac{\theta_1(x)+\theta_2(x)}{2}(u^+)^{p-1}\varphi\,dx.\]
So, $K_\infty:\w\to W^{-s,p'}(\Omega)$ is a completely continuous map. We define a continuous $(S)_+$-homotopy $h_\infty:[0,1]\times\w\to W^{-s,p'}(\Omega)$ by setting for all $(t,u)\in[0,1]\times\w$
\[h_\infty(t,u) = \fpl u-(1-t)N_f(u)-tK_\infty(u).\]
We prove next that for all $R>0$ big enough, $t\in[0,1]$, and $\|u\|=R$
\beq\label{nco5}
h_\infty(t,u) \neq 0.
\eeq
Arguing by contradiction, assume that there exist sequences $(t_n)$ in $[0,1]$ and $(u_n)$ in $\w$, s.t.\ $\|u_n\|\to\infty$ and for all $n\in\N$
\[h_\infty(t_n,u_n) = 0.\]
Passing if necessary to a subsequence, we have $t_n\to t$ for some $t\in[0,1]$. Also, testing the relation above with $-u_n^-\in\w$ and recalling that $f(\cdot,t)=0$ in $\Omega$ for all $t\le 0$, we have for all $n\in\N$
\begin{align*}
\|u_n^-\|^p &\le \langle\fpl u_n,-u_n^-\rangle \\
&= (1-t_n)\int_{\{u_n<0\}} f(x,u_n)u_n\,dx+t_n\int_{\{u_n<0\}}\frac{\theta_1(x)+\theta_2(x)}{2}(u_n^+)^{p-1}u_n\,dx = 0,
\end{align*}
so $u_n\ge 0$ in $\Omega$. Clearly we may assume $u_n\neq 0$, so set
\[v_n = \frac{u_n}{\|u_n\|} \in\w_+\setminus\{0\}.\]
For all $n\in\N$ we have weakly in $\Omega$
\beq\label{nco6}
\fpl v_n = (1-t_n)\frac{f(x,u_n)}{\|u_n\|^{p-1}}+t_n\frac{\theta_1(x)+\theta_2(x)}{2}v_n^{p-1}.
\eeq
Clearly $\|v_n\|=1$, so up to a subsequence we have $v_n\rightharpoonup v$ in $\w$, $v_n\to v$ in $L^p(\Omega)$, and $v_n(x)\to v(x)$ for a.e.\ $x\in\Omega$ with $L^p(\Omega)$-dominated convergence. Hence $v\ge 0$ in $\Omega$. Testing \eqref{nco6} with $(v_n-v)\in\w$, using \eqref{nco3} and H\"older's inequality, we get for all $n\in\N$
\begin{align*}
\langle\fpl v_n,v_n-v\rangle &= (1-t_n)\int_\Omega\frac{f(x,u_n)}{\|u_n\|^{p-1}}(v_n-v)\,dx+t_n\int_\Omega\frac{\theta_1(x)+\theta_2(x)}{2}v_n^{p-1}(v_n-v)\,dx \\
&\le C\frac{1+\|u_n\|_p^{p-1}}{\|u_n\|^{p-1}}\|v_n-v\|_p+C\|v_n\|_p^{p-1}\|v_n-v\|_p \le C\|v_n-v\|_p,
\end{align*}
and the latter tends to $0$ as $n\to\infty$. By the $(S)_+$-property of $\fpl$, we have $v_n\to v$ in $\w$, in particular $\|v\|=1$. By \eqref{nco3} we have for all $n\in\N$
\[\int_\Omega\Big|\frac{f(x,u_n)}{\|u_n\|^{p-1}}\Big|^{p'}\,dx \le C\frac{1+\|u_n\|_p^p}{\|u_n\|^p} \le C.\]
By reflexivity, we can find $g_\infty\in L^{p'}(\Omega)$ s.t.\ up to a subsequence we have in $L^{p'}(\Omega)$
\[\frac{f(\cdot,u_n)}{\|u_n\|^{p-1}} \rightharpoonup g_\infty.\]
Fix $\eps>0$ and set for all $n\in\N$
\[\Omega^\eps_n = \Big\{x\in\Omega:\,u_n(x)>0, \ \theta_1(x)-\eps\le\frac{f(x,u_n(x))}{u_n(x)^{p-1}}\le\theta_2(x)+\eps\Big\}.\]
Also set
\[\Omega^+ = \big\{x\in\Omega:\,v(x)>0\big\}.\]
Then we have $u_n(x)=\|u_n\|v_n(x)\to\infty$ for a.e.\ $x\in\Omega^+$, so by ${\bf H}_2$ \ref{h22} we have $\chi_{\Omega^\eps_n}\to 1$ in $\Omega^+$, with dominated convergence. So we have in $L^{p'}(\Omega^+)$
\beq\label{nco7}
\chi_{\Omega^\eps_n}\frac{f(\cdot,u_n)}{\|u_n\|^{p-1}} \rightharpoonup g_\infty.
\eeq
Besides, by definition of $\Omega^\eps_n$, for all $n\in\N$ we have in $\Omega^+$
\[\chi_{\Omega^\eps_n}(\theta_1-\eps)v_n^{p-1} \le \chi_{\Omega^\eps_n}\frac{f(\cdot,u_n)}{\|u_n\|^{p-1}} \le \chi_{\Omega^\eps_n}(\theta_2+\eps)v_n^{p-1}.\]
By Mazur's theorem (see \cite[Corollary 3.8]{B}), up to a convex combination we have strong convergence in \eqref{nco7}, so we can pass to the limit as $n\to\infty$ in the inequalities above and find in $\Omega^+$
\[(\theta_1-\eps)v^{p-1} \le g_\infty \le (\theta_2+\eps)v^{p-1}.\]
Let $\eps\to 0^+$ to find in $\Omega^+$
\[\theta_1v^{p-1} \le g_\infty \le \theta_2v^{p-1}.\]
Besides, in $\Omega	\setminus\Omega^+$ we have $v_n\to 0$. By \eqref{nco3} and the definition of $v_n$, for all $n\in\N$ we have in $\Omega\setminus\Omega^+$
\[\Big|\frac{f(x,u_n)}{\|u_n\|^{p-1}}\Big| \le C\frac{1+u_n^{p-1}}{u_n^{p-1}}v_n^{p-1} \le Cv_n^{p-1},\]
and the latter tends to $0$ as $n\to\infty$. So we may set $g_\infty=0$ in $\Omega\setminus\Omega^+$. In conclusion, there exists $\theta\in L^\infty(\Omega)$ s.t.\ in $\Omega$
\[\theta_1 \le \theta \le \theta_2, \ g_\infty = \theta v^{p-1}.\]
Besides, we have in $L^{p'}(\Omega)$
\[\frac{\theta_1+\theta_2}{2}v_n^{p-1} \to \frac{\theta_1+\theta_2}{2}v^{p-1}.\]
Define now
\[m_\infty = (1-t)\theta+t\frac{\theta_1+\theta_2}{2} \in L^\infty(\Omega).\]
Passing to the limit in \eqref{nco6} as $n\to\infty$, we have weakly in $\Omega$
\[\fpl v = m_\infty(x)v^{p-1}.\]
So $v\neq 0$ is an eigenfunction of problem \eqref{evp} with weight $m=m_\infty$, associated to the eigenvalue $1$. By ${\bf H}_2$ \ref{h22} and the computation above we have $m_\infty\ge\lambda_1$ in $\Omega$, $m_\infty\neq\lambda_1$, so by Proposition \ref{pev} \ref{pev2} we have
\[\lambda_1(m_\infty) < \lambda_1(\lambda_1) = 1.\]
Thus, $v$ is a non-principal eigenfunction, which along with Proposition \ref{pev} \ref{pev1} implies that $v$ is nodal, a contradiction. So \eqref{nco5} is proved.
\vskip2pt
\noindent
Using \eqref{nco5} and Proposition \ref{deg} \ref{deg4} (homotopy invariance), we see that for all $R>0$ big enough
\beq\label{nco8}
{\rm deg}_{(S)_+}(\Phi',B_R(0),0) = {\rm deg}_{(S)_+}(\fpl-K_\infty,B_R(0),0).
\eeq
In order to evaluate the latter quantity, we shall define a new homotopy. Let $w\in C^\infty_c(\Omega)_+\setminus\{0\}$ be identified with an element of $W^{-s,p'}(\Omega)$, and set for all $(t,u)\in[0,1]\times\w$
\[\tilde h_\infty(t,u) = \fpl u-K_\infty(u)-tw.\]
Clearly $\tilde h_\infty:[0,1]\times\w\to W^{-s,p'}(\Omega)$ is a continuous $(S)_+$-homotopy. Moreover, for all $t\in[0,1]$ and all $u\in\w\setminus\{0\}$ we have
\beq\label{nco9}
\tilde h_\infty(t,u) \neq 0.
\eeq
Arguing by contradiction, let $\tilde h_\infty$ vanish at some $t\in[0,1]$, $u\in\w\setminus\{0\}$. Then we have weakly in $\Omega$
\beq\label{nco10}
\fpl u = \frac{\theta_1(x)+\theta_2(x)}{2}(u^+)^{p-1}+tw(x).
\eeq
Testing \eqref{nco10} with $-u^-\in\w$ we have
\begin{align*}
\|u^-\|^p &\le \langle\fpl u,-u^-\rangle &\\
&= \int_{\{u<0\}}\frac{\theta_1(x)+\theta_2(x)}{2}(u^+)^{p-1}u\,dx+t\int_{\{u<0\}}w(x)u\,dx \le 0,
\end{align*}
so $u\ge 0$ in $\Omega$. Now set
\[m = \frac{\theta_1+\theta_2}{2} \in L^\infty(\Omega),\]
and note that by ${\bf H}_2$ \ref{h22} $m\ge\lambda_1$ in $\Omega$, with $m\neq\lambda_1$. By Proposition \ref{pev} \ref{pev2} we have
\[\lambda_1(m) < \lambda_1(\lambda_1) = 1.\]
Now we distinguish two cases:
\begin{itemize}[leftmargin=1cm]
\item[$(a)$] If $t=0$, then by \eqref{nco10} we see that $u$ is an eigenfunction of problem \eqref{pev} with weight $m$, associated to the non-principal eigenvalue $1$, hence $u$ must be nodal, a contradiction.
\item[$(b)$] If $t\in(0,1]$, then we apply Proposition \ref{amp} with $m$ as above, $\lambda=1$, and $\beta=tw$, and by \eqref{nco10} we deduce that $u^-\neq 0$, a contradiction again.
\end{itemize}
So \eqref{nco9} is proved. By Proposition \ref{deg} \ref{deg4} (homotopy invariance) we have for all $R>0$
\beq\label{nco11}
{\rm deg}_{(S)_+}(\fpl-K_\infty,B_R(0),0) = {\rm deg}_{(S)_+}(\fpl-K_\infty-w,B_R(0),0).
\eeq
In addition, by \eqref{nco9} and Proposition \ref{deg} \ref{deg5} (solution) we have for all $R>0$
\[{\rm deg}_{(S)_+}(\fpl-K_\infty-w,B_R(0),0) = 0.\]
Concatenating \eqref{nco8}, \eqref{nco11}, and the above relation we achieve the proof of \eqref{nco4}.
\vskip2pt
\noindent
Now we study the behavior of $\Phi'$ near the origin, where for all $\rho>0$ small enough we have
\beq\label{nco12}
{\rm deg}_{(S)_+}(\Phi',B_\rho(0),0) = 0.
\eeq
Due to the similarity of hypotheses ${\bf H}_2$ \ref{h22} \ref{h23}, we may prove \eqref{nco12} with an analogous argument to that used for \eqref{nco4}, so we omit the details. Set for all $u,\varphi\in\w$
\[\langle K_0(u),\varphi\rangle = \int_\Omega\frac{\eta_1(x)+\eta_2(x)}{2}(u^+)^{p-1}\varphi\,dx,\]
with $\eta_1,\eta_2\in L^\infty(\Omega)_+$ as in ${\bf H}_2$ \ref{h23}. Also, set for all $(t,u)\in[0,1]\times\w$
\[h_0(t,u) = \fpl u-(1-t)N_f(u)-tK_0(u).\]
Then, $h_0:[0,1]\times\w\to W^{-s,p'}(\Omega)$ is a continuous $(S)_+$-homotopy s.t.\ or all $t\in[0,1]$ and all $\|u\|>0$ small enough
\[h_0(t,u) \neq 0.\]
By homotopy invariance we have for all $\rho>0$ small enough
\[{\rm deg}_{(S)_+}(\Phi',B_\rho(0),0) = {\rm deg}_{(S)_+}(\fpl-K_0,B_\rho(0),0).\]
In addition, fix as before $w\in C^\infty_c(\Omega)_+\setminus\{0\}$ and set for all $(t,u)\in[0,1]\times\w$
\[\tilde h_0(t,u) = \fpl u-K_0(u)-tw.\]
So we have for all $t\in[0,1]$ and all $u\in\w\setminus\{0\}$
\[\tilde h_0(t,u) \neq 0.\]
By the homotopy invariance and solution properties, then, for all $\rho>0$
\[{\rm deg}_{(S)_+}(\fpl-K_0,B_\rho(0),0) = {\rm deg}_{(S)_+}(\fpl-K_0-w,B_\rho(0),0) = 0.\]
Concatenating the relations above, we get \eqref{nco12}.
\vskip2pt
\noindent
Finally, we go back to the solution $u_1\in\w_+$ detected in \eqref{nco1}, recalling that it is a local minimizer of $\Phi$. Once again we distinguish two cases:
\begin{itemize}[leftmargin=1cm]
\item[$(a)$] If $u_1$ is not an isolated critical point of $\Phi$, then clearly there exist infinitely many nontrivial critical points of such functional.
\item[$(b)$] If $u_1$ is an isolated critical point of $\Phi$, then it is in particular a {\em strict} local minimizer. So, by Proposition \ref{rab} \ref{rab1} we have for all $\rho>0$ small enough
\beq\label{nco13}
{\rm deg}_{(S)_+}(\Phi',B_\rho(u_1),0) = 1.
\eeq
Comparing \eqref{nco12} and \eqref{nco13} we see that $u_1\neq 0$. Also, fixing $\rho>0$ small enough and $R>0$ big enough we may have
\[\overline{B}_\rho(0)\cap\overline{B}_\rho(u_1) = \emptyset, \ \overline{B}_\rho(0)\cup\overline{B}_\rho(u_1) \subset B_R(0),\]
as well as \eqref{nco4} \eqref{nco12} \eqref{nco13}. Applying twice Proposition \ref{deg} \ref{deg2} (domain additivity) we get
\begin{align*}
0 &= {\rm deg}_{(S)_+}(\Phi',B_R(0),0) \\
&= {\rm deg}_{(S)_+}(\Phi',B_\rho(0),0)+{\rm deg}_{(S)_+}(\Phi',B_\rho(u_1),0)+{\rm deg}_{(S)_+}\big(\Phi',B_R(0)\setminus(\overline{B}_\rho(0)\cup\overline{B}_\rho(u_1)),0\big) \\
&= 1+{\rm deg}_{(S)_+}\big(\Phi',B_R(0)\setminus(\overline{B}_\rho(0)\cup\overline{B}_\rho(u_1)),0\big),
\end{align*}
that is,
\[{\rm deg}_{(S)_+}\big(\Phi',B_R(0)\setminus(\overline{B}_\rho(0)\cup\overline{B}_\rho(u_1)),0\big) = -1.\]
By Proposition \ref{deg} \ref{deg5} (solution), there exists $u_2\in B_R(0)\setminus(\overline{B}_\rho(0)\cup\overline{B}_\rho(u_1))$ s.t.\
\[\Phi'(u_2) = 0.\]
\end{itemize}
In either case, there exist (at least) two nontrivial critical points $u_1,u_2\in K(\Phi)\setminus\{0\}$. Choose $i\in\{1,2\}$ and argue as before. By Proposition \ref{reg} we have $u\in C^\alpha_s(\overline\Omega)$ satisfies weakly in $\Omega$
\[\fpl u_i = f(x,u_i).\]
Testing with $-u_i\in\w$ we have as above
\begin{align*}
\|u_i^-\|^p &\le \langle\fpl u_i,-u_i^-\rangle \\
&= \int_{\{u_i<0\}}f(x,u_i)u_i\,dx = 0,
\end{align*}
hence $u_i\ge 0$ in $\Omega$. In addition, by ${\bf H}_2$ \ref{h21} \ref{h23} we have for a.e.\ $x\in\Omega$ and all $t\in[0,\|u_i\|_\infty]$
\[f(x,t) \ge -Ct^{p-1}.\]
So we have weakly in $\Omega$
\[\fpl u_i+Cu_i^{p-1} \ge 0.\]
Applying Proposition \ref{scp} \ref{scp1} (with $g(t)=(t^+)^{p-1}$), we deduce that $u_i\in{\rm int}(\cs)$, which concludes the proof.
\end{proof}

\begin{remark}\label{neg}
Under hypotheses symmetric to ${\bf H}_2$, we can prove that \eqref{dir} admits two negative solutions.
\end{remark}

\vskip4pt
\noindent
{\bf Acknowledgement.} The author is a member of GNAMPA (Gruppo Nazionale per l'Analisi Matematica, la Probabilit\`a e le loro Applicazioni) of INdAM (Istituto Nazionale di Alta Matematica 'Francesco Severi'). He is partially supported by the research projects: {\it Evolutive and Stationary Partial Differential Equations with a Focus on Biomathematics} (Fondazione di Sardegna 2019), {\em Analysis of PDEs in connection with real phenomena} (CUP F73C22001130007, Fondazione di Sardegna 2021), {\em Nonlinear Differential Problems via Variational, Topological and Set-valued Methods} (PRIN-2017AYM8XW). The authors wishes to thank the anonymous Referee for her/his careful reading of the work and precious suggestions.

\end{document}